\documentclass[10pt,a4paper, openany,  oldfontcommands]{memoir}
                                                      \usepackage[utf8]{inputenc}
                                                      \usepackage[english]{babel}
                                                      \usepackage{amsmath}
                                                      \usepackage{amsthm}
                                                      \usepackage{thmtools}
                                                      \usepackage{mathtools}
                                                      \usepackage{amsfonts}
                                                      \usepackage{amssymb}
                                                      \usepackage{makeidx}
                                                      \usepackage{graphicx}
                                                      \usepackage{tikz-cd}
                                                      \usepackage{float}
                                                      \usepackage{authblk}

                                                      \usepackage[linguistics]{forest}

                                                      \usepackage{hyperref}
                                                      \usepackage{doi}

                                                      \usepackage{xcolor}

                                                      \usepackage{stmaryrd}

                                                      \usepackage{braket}

                                                         \title{Friedman's  $  \mathsf{WD} $  is not parameter-free sequential}

                                                        \author{Juvenal Murwanashyaka}
                                                        \affil{Institute of Mathematics,  Czech Academy of Sciences,   Czech Republic}

                                                      \newtheorem{theorem}{Theorem} 
                                                      \newtheorem{lemma} [theorem] {Lemma}
                                                       \newtheorem{corollary} [theorem] {Corollary}
                                                      \newtheorem{definition} [theorem] {Definition} 
                                                       
                                                       \newtheorem{open problem} [theorem] {Open Problem}

                                                      \setsecnumdepth{subsubsection} 
                                                      \settocdepth{subsubsection}

\begin{document}

\maketitle 
 
\begin{abstract}
Harvey Friedman's $  \mathsf{WD} $  is  a  weak set  theory  given by the following  non-logical  axioms:
\textsf{(W)}   $  \forall x y   \,  \exists z  \,    \forall u   \left[    \,   u  \in z  \leftrightarrow   (    \,    u \in x   \;  \vee   \;    u  =  y     \,    )      \,      \right]    $; 
\textsf{(D)}  $   \forall x y   \,  \exists z  \,    \forall u   \,   \left[    \,    u  \in z  \leftrightarrow   (    \,    u \in x   \;  \wedge   \;    u  \neq  y     \,    )      \,       \right]    $. 
We answer   a   question  raised  by Albert Visser which  asks whether  $  \mathsf{WD} $  is  parameter-free sequential. 
Let  $  \mathsf{WD} +  \mathsf{EXT} $ denote the theory we obtain by extending  $  \mathsf{WD} $ with  the axiom of extensionality. 
We show that   $   \mathsf{WD} +  \mathsf{EXT}  $, and hence also  $   \mathsf{WD} $, 
  is not parameter-free sequential by using forcing  to  construct a model $    \mathcal{V}^{  \star }  $   of $     \mathsf{WD} +  \mathsf{EXT}   $      where 
$  \left(   \mathcal{V}^{  \star }  ,  a   \right)   \simeq   \left(   \mathcal{V}^{  \star }  ,   b   \right)    $  
 for any two elements  $a, b $  of  $  \mathcal{V}^{  \star }  $. 
\end{abstract}

\section{Introduction  }

Adjunctive set theory   $  \mathsf{AS} $    is  a  weak set theory given by the following two non-logical axioms: 
\[
\begin{array}{r l  c  c r l  }
\mathsf{AS}_1 
& 
 \exists x  \;  \forall y   \;   [    \   y \not\in x    \   ]   
\\
\mathsf{AS}_2 
& 
 \forall x y   \;  \exists z  \;    \forall u   \;   [    \   u \in z  \leftrightarrow   (    \    u \in x   \;  \vee   \;    u  =  y     \    )      \      ]    
 \end{array}
\] 
The structure  $  \left(  V_{  \omega }   \,  ,  \,   \in  \right)  $  of  hereditarily finite sets  is  a minimal model of  $  \mathsf{AS} $. 
The axioms of  $  \mathsf{AS} $  go back to   Bernays     \cite{Barnays1937}. 
 In \cite{Pudlak1983},    Pudlák  uses $  \mathsf{AS} $  for  a formal characterization of   sequential   theories.
 Sequential  theories are   theories with a coding machinery, for  all objects in the domain of the theory,    sufficient for developing 
partial  satisfaction predicates for formulas of bounded depth-of-quantifier-alternations  (see Visser \cite{Visser2019}). 
Formally, a first-order theory  $T$ is  \emph{sequential}   if it  $1$-directly interprets  $  \mathsf{AS} $, that is, 
there exists a  first-order formula  $  \phi ( x, y, z_0,  \ldots , z_n ) $  in the language of  $T$, 
with all free variable displayed,  such that   for any model $  \mathcal{M} \models  T $   with universe $  M$,  
there exist  $  p_0,  \ldots , p_n  \in M  $  such that 
   $  ( M, R_{  \vec{p}  } )  \models    \mathsf{AS}  $  where 
   $  R_{  \vec{p} }  :=  \lbrace (u, v )  \in  M^2  :   \   \mathcal{M} \models  \phi  (u, v , p_0,  \ldots  , p_n  )  \,  \rbrace     \,  $.

  Friedman \cite{Friedman2007}[p.~17] has introduced a variant $  \mathsf{WD} $  given by the following two non-logical axioms: 
\[
\begin{array}{r l  c  c r l  }
\mathsf{W}
& 
\forall x y   \,  \exists z  \,    \forall u   \left[    \,   u  \in z  \leftrightarrow   (    \,    u \in x   \;  \vee   \;    u  =  y     \,    )      \,      \right] 
\\
\mathsf{D} 
& 
  \forall x y   \,  \exists z  \,    \forall u   \left[     \,   u \in z  \leftrightarrow   (    \,    u  \in x   \;  \wedge   \;    u  \neq  y     \,    )      \,      \right]       
 \end{array}
\] 
  Friedman \cite{Friedman2007}[Theorem 3.1]  uses $  \mathsf{WD} $  to give an alternative characterization of sequential theories: 
   a first-order theory  $T$ is   sequential   if and only if  it $1$-directly interprets  $  \mathsf{WD} $.
In particular,  $  \mathsf{WD} $ is sequential: 
given  a model  $   \left( M,  \in  \right)   \models  \mathsf{WD} $  and  $  p  \in  M  $, 
we have  $  \left(  M,  R_p \right)   \models   \mathsf{AS} $  where  
\[
  R_p =  \lbrace    (x, y )  \in  M^2  :   \    \left( M,  \in  \right)   \models    x  \in  y  \leftrightarrow   x  \not\in p    \,   \rbrace      
\       .
\]


In this paper, we show that the two  weak set theories  $  \mathsf{AS} $  and  $  \mathsf{WD}  $ are rather different: 
we   construct a model   $  \mathcal{V}^{  \star }  =  \left(   V^{  \star }   ,  \in^{  \star }  \right)   $  of  $  \mathsf{WD} $
where for any two elements  $ a, b  \in V^{  \star  }  $  there exists an automorphism $  F :   \mathcal{V}^{  \star }   \to   \mathcal{V}^{  \star }   $  such that  $  F(a) = b $. 
In contrast, for any automorphism  $  G $  of a model  $  \mathcal{M} $ of  $  \mathsf{AS} $, 
we have $  G( a )   \not\in  V_{  \omega }  \setminus  \lbrace a  \rbrace  $  for all $  a  \in V_{  \omega  }   $. 
In particular,  $  \left(  V^{  \star }   ,  R   \right)   \not\models  \mathsf{AS}  $ 
 for  any binary relation $  R  \subseteq  V^{  \star }  \times  V^{  \star }  $  that is 
 first-order definable  in   $     \mathcal{V}^{  \star }   $  without parameters. 
This  answers   a  question  of     Visser \cite{Visser2010}[p.~243]  which asks whether    $  \mathsf{WD} $ is a parameter-free sequential theory.
A first-order theory  $T$ is \emph{parameter-free sequential} if there exists a  first-order formula   $  R (x, y ) $  in the language of  $T$, with only $x, y $ free,  such that  $  \left(  M,   R^{  \mathcal{M} }   \right)   \models  \mathsf{AS}  $  for all  $  \mathcal{M} \models  T $   with universe $  M  $; 
we write  $  R^{  \mathcal{M} }   $  for  $  \lbrace  (x,y )  \in M^2  :  \   \mathcal{M} \models  R(x, y )   \,  \rbrace  $.
The reasoning above shows that  the existence of the model    $     \mathcal{V}^{  \star }   $ implies   that  $  \mathsf{WD} $  is not  parameter-free sequential. 
The theory   $  \mathsf{WD} $   is thus    an essentially parametrically sequential theory, 
which means the theory  needs a parameter to witness its sequentiality.

We will   prove a slightly  stronger statement. 
Let  $  \mathsf{EXT} $  denote the axiom of extensionality 
$  \forall x  y   \left[  \,   \forall u  \left[   \,  u  \in x  \leftrightarrow  u  \in y   \,   \right]    \rightarrow  x  =  y    \,    \right]     \,    $. 
Let  $  \mathsf{BU} $ denote the  binary union axiom 
$
 \forall x y  \,  \exists z   \,    \forall u    \left[    \   u  \in z  \leftrightarrow     (   \,   u \in x   \;   \vee  \;  u  \in y    \,  )      \       \right]       
\,      $. 
Let  $  \mathsf{BI} $ denote the  binary intersection axiom 
$
 \forall x y  \,  \exists z   \,    \forall u    \left[    \   u  \in z  \leftrightarrow     (   \,   u \in x   \;   \wedge  \;  u  \in y    \,  )      \       \right]       
\,      $. 
We  prove the following theorem.

\begin{theorem}  \label{maintheorem}
There exists a model  $  \mathcal{V}^{ \star }    $    of   
 $  \mathsf{WD}  +  \mathsf{EXT}   +   \mathsf{BU}  +  \mathsf{BI}$
 such that   $  \left(   \mathcal{V}^{  \star }  ,  a   \right)   \simeq   \left(   \mathcal{V}^{  \star }  ,   b   \right)    $  
 for any two elements  $a, b $  of  $  \mathcal{V}^{  \star }  $. 
\end{theorem}

\begin{corollary}  
 $  \mathsf{WD} +   \mathsf{EXT}    +   \mathsf{BU}  +  \mathsf{BI}$ is not  parameter-free  sequential. 
\end{corollary}

The rest of the paper is devoted to proving  Theorem  \ref{maintheorem}.

\section{The forcing partial order}

We prove Theorem  \ref{maintheorem}   by using forcing to construct a countable  model
    $   \mathcal{V}^{ \star }  =    \left(   V^{  \star }   ,  \in^{  \star }  \right)       $  of  
    $   \mathsf{WD}  +   \mathsf{EXT}    +   \mathsf{BU}  +  \mathsf{BI}  $
   such that     $  \left(   \mathcal{V}^{  \star }  ,  a   \right)   \simeq   \left(   \mathcal{V}^{  \star }  ,   b   \right)    $   
   for all $ a, b  \in  V^{  \star }      $; 
   we  remark  to the reader  that, although we use forcing terminology, this is not a relative consistency proof. 
We will ensure that  for all $  a, b  \in   V^{  \star }  $, 
   we can use  a  back-and-forth argument    to construct an automorphism  $F   :    \mathcal{V}^{ \star }   \to  \mathcal{V}^{ \star }    $
 such that  $  F \left(  a  \right)  =  b  $. 
Since  we will have   $  \left(   \mathcal{V}^{  \star }  ,  a   \right)   \simeq   \left(   \mathcal{V}^{  \star }  ,   b   \right)    $     for all $ a, b  \in  V^{  \star }      $,  one of the following  must  hold: 
(1)   $  \mathcal{V}^{  \star }    \models    \forall x  \left[  \,   x  \in  x  \,   \right]   $; 
(2)   $      \mathcal{V}^{  \star }    \models     \forall x  \left[  \,   x    \not\in  x  \,   \right]      \,    $.
 We choose to  construct a model where   $   \mathcal{V}^{ \star }   \models  \forall x  \left[  \,   x  \in  x  \,   \right]        $.

We proceed to construct the model       $   \mathcal{V}^{ \star }    $.
As usual, let  $  \mathbb{Z} $  denote the set of integers, and let $  \omega $ denote the set of nonnegative integers. 
For each  $ k \in  \mathbb{Z}  $,    pick a unique constant symbol $  \mathsf{c}_k $.  
The universe of   $  \mathcal{V}^{  \star }  $   will be  the set   
\[
  V^{  \star }  :=  \big\{    \mathsf{c}_k   :   \   k  \in \mathbb{Z}      \,    \big\}  
  \        .
 \] 
Let   $  \mathcal{X} $  denote the set of all subsets of   $    V^{  \star }  $  that differ from  
$   \lbrace    \mathsf{c}_k   :   \   k  \in \omega   \rbrace   $   by finitely many elements, that is 
\[
 \mathcal{X}  :=  
 \bigg\{    X  \subseteq  V^{  \star }    :   \  
    \vert    \left(  X  \setminus  \lbrace    \mathsf{c}_k   :   \   k  \in \omega   \rbrace   \right)  
\cup   
 \left(    \lbrace    \mathsf{c}_k   :   \   k  \in \omega     \rbrace    \setminus X \right)     \vert    <   \aleph_0 
 \      \bigg\}
 \          .
\]
The membership relation  $  \in^{  \star }  $  on   $  V^{  \star }    $ will be such that each element of $  V^{  \star }   $ realizes  a unique  set in   $   \mathcal{X}  $. 
That is, we construct a  bijection        $   ( \cdot )^{ \star }   :     V^{  \star }    \to   \mathcal{X}    $
and  define the  relation   $  \in^{  \star }  $ as follows: 
 \[
    \mathsf{c}_i   \in^{  \star }   \mathsf{c}_j  
\    \Leftrightarrow_{  \mathsf{def}  }     \       
      \mathsf{c}_i    \in   \mathsf{c}_j^{  \star }  
      \     \     \mbox{  for all }   i, j  \in  \mathbb{Z} 
      \        .
\]

\begin{lemma}   \label{FirstBasicLemma}
Let   $  (  \cdot )^{  \star }   :  V^{  \star }     \to        \mathcal{X}   $    be a  bijection.  
For all  $i, j \in   \mathbb{Z} $, let    
\[
   \mathsf{c}_j   \in^{   \star }   \mathsf{c}_k       \        \     \Leftrightarrow_{  \mathsf{def}  }      \       \       \mathsf{c}_j   \in    \mathsf{c}_k^{  \star }   
   \      .
   \]
Then, $ \left(  V^{  \star }    ,  \in^{ \star }  \right)    \models      \mathsf{WD} +  \mathsf{EXT} +       \mathsf{BU}    +  \mathsf{BI} $. 
\end{lemma}
\begin{proof}
The axiom of extensionality holds since  $    (  \cdot )^{  \star }  $  is a bijection. 
The axioms of  $  \mathsf{WD}  $ hold since $  \mathcal{X} $ is closed under adjunction and subtraction of elements of  $  V^{  \star }  $,  that is, 
for all   $  X \in  \mathcal{X}  $  and  all  $ k \in   \mathbb{Z} $,  we have 
$  X  \cup  \lbrace \mathsf{c}_k \rbrace \in  \mathcal{X}    $  and  $  X  \setminus   \lbrace \mathsf{c}_k \rbrace \in  \mathcal{X}    $.
The binary union axiom and the binary intersection axiom hold  since   $   \mathcal{X}   $  is closed under  binary unions and binary intersections, that is, 
for all $ X, Y \in    \mathcal{X}     $,   we have $  X \cup Y \in   \mathcal{X}     $  and $  X  \cap  Y  \in    \mathcal{X}     $.
\end{proof}

It remains to construct a suitable  bijection    $  (  \cdot )^{  \star }   :  V^{  \star }     \to        \mathcal{X}   $. 
We construct the map  $  (  \cdot )^{  \star }    $ by recursion as   the union of a suitable chain in 
 the   partial order $  \left(   \mathbb{P}   ,  \subseteq       \right)    $   defined as follows: 
 \begin{enumerate}
 \item    $   \mathbb{P}   $   is the set  of all partial one-to-one maps   $    \sigma :  V^{  \star }  \to   \mathcal{X}   $  such that: 
 \begin{enumerate}
\item    $    u  \in  \sigma (u)     $   for all  $   u   \in   \mathsf{dom} \left(  \sigma  \right)        \,     $.

 \item     $ \cap    \sigma  \left[   \mathsf{dom} \left(  \sigma  \right)    \right]    \setminus    \mathsf{dom} \left(  \sigma  \right)    $
 and  
   $ V^{  \star }  \setminus   \cup    \sigma  \left[   \mathsf{dom} \left(  \sigma  \right)    \right]    $
 are both infinite. 
  \end{enumerate}

 \item  $  \sigma  \subseteq  \tau $   if and only if  $  \sigma $ is a restriction of  $  \tau $. 
\end{enumerate}  
Clause (1a)   implies that  we will get a model where    $    \forall x  \left[  \,   x  \in  x  \,  \right]       $   holds. 
The first conjunct of (1b) implies that   $   \mathcal{X} \setminus  \sigma  \left[   \mathsf{dom} \left(  \sigma  \right)    \right]  $  is infinite: 
 for  each $ x  \in   \cap    \sigma  \left[   \mathsf{dom} \left(  \sigma  \right)    \right]    $, 
the set  $  \lbrace  \mathsf{c}_k :  \   k  \in  \omega   \,  \rbrace   \setminus  \lbrace x  \rbrace   $  
is not in the image of  $  \sigma   $  
since  $  x  \in  \sigma (u )   $   for all $  u  \in     \mathsf{dom} \left(  \sigma  \right)      \,     $. 
Similarly, the second   conjunct of (1b) implies that   
$   \mathcal{X} \setminus  \sigma  \left[   \mathsf{dom} \left(  \sigma  \right)    \right]  $   is infinite: 
 for  each $ x  \in   V^{  \star }  \setminus   \cup    \sigma  \left[   \mathsf{dom} \left(  \sigma  \right)    \right]    $, 
the set  $  \lbrace  \mathsf{c}_k :  \   k  \in  \omega   \,  \rbrace   \cup  \lbrace x  \rbrace   $  
is not in the image of  $  \sigma  $
since  $  x   \not\in  \sigma (u )   $   for all $  u  \in     \mathsf{dom} \left(  \sigma  \right)      \,     $.

Clause (1b) is a  technical requirement that  has  to do with the fact that we will  mainly be concerned  with $  \sigma \in  \mathbb{P} $  where  $  \mathsf{dom} \left(  \sigma  \right)  $  is infinite. 
The reason for this is that   the maps we will  work with will be of the form   $  f  :  W   \to   V^{  \star }  $
where $  W  \subseteq    \mathsf{dom} \left(  \sigma  \right)    $  is such that: 
\begin{itemize}
\item[(i)]  $  \sigma (x)  \setminus  \sigma (y)   \subseteq   W  $  for all $  x, y  \in W  $;

\item[(ii)]  for all   $  u  \in W $  and all  finite  sets $  A, B  \subseteq  W  $,  there exists  $  v \in  W  $  such that  
$  \sigma (v)  =  \left(   \sigma (u )  \setminus  A   \right)   \cup   B  $. 
\end{itemize}
 Clause (1b) will ensure that we can extend $  \sigma  $  to  $  \tau  \in  \mathbb{P} $ in  such a way that  certain maps of the form (i)-(ii) exist.

We will refer to elements of   $  \mathbb{P}  $  as  \emph{conditions}. 
We will call a set $D  \subseteq     \mathbb{P}   $   \emph{dense} if for each  $\sigma \in  \mathbb{P}    $  there exists  $ \tau  \in   D  $ 
such that $  \sigma  \subseteq    \tau $. 
Given a   family $  \Gamma $  of dense subsets of    $ \mathbb{P}    $, 
we  will say that  a  nonempty set  $  G  \subseteq   \mathbb{P}   $ is a   \emph{$  \Gamma $-generic ideal}  if: 
\begin{enumerate}
\item    $G$ is downward closed under $  \subseteq    $,  that is, 
$  \left(  \forall  \tau   \in G  \right)   \left(   \forall  \sigma  \in  \mathbb{P} \right)  \left[  
 \sigma  \subseteq   \tau    \rightarrow    \sigma  \in G    \,   \right]  
$.  
\item $G$ is upwards directed,  that is,   for all  $  \sigma_0,  \sigma_1  \in G $,   there exists  $  \tau \in G $ such that 
$  \sigma_0   \subseteq    \tau $  and  $  \sigma_1 \subseteq    \tau $; 

\item  $  D  \cap  G \neq  \emptyset $  for   all $  D  \in \Gamma $. 
\end{enumerate}
If $  \Gamma$ is countable, 
then we can  construct  a $  \Gamma$-generic ideal  $ G $  by recursion.

We   complete the construction of  the  bijection     $  (  \cdot )^{  \star }   :  V^{  \star }     \to        \mathcal{X}   $ by defining 
 a countable  family  $  \mathcal{D} $ of dense subsets of  $  \mathbb{P} $  that correspond to properties  we need the map  to have, 
 and  then set   $ (  \cdot )^{  \star }  :=   \cup G $  where  $G$ is any  $  \mathcal{D} $-generic  ideal. 
 The family  $  \mathcal{D} $  will be a union  $   \mathcal{D}   =   \mathcal{T}  \cup   \mathcal{S}  \cup  \mathcal{G} $  where: 
 \begin{enumerate}
 
\item   $  \mathcal{T}  $   consist of dense subsets of $  \mathbb{P} $ that will ensure that 
$ (  \cdot )^{  \star }    :  V^{  \star }  \to  \mathcal{X} $ is a total function.

 \item   $  \mathcal{S}  $   consist of dense subsets of $  \mathbb{P} $ that will ensure that 
 $ (  \cdot )^{  \star }    :  V^{  \star }  \to  \mathcal{X} $ is surjective and hence a bijection.

\item  $  \mathcal{G}    $    consist of dense subsets of $  \mathbb{P} $ that will ensure that    for all   $k, \ell   \in   \mathbb{Z} $ we can use a  back-and-forth argument  to construct an  automorphism 
 $  F^k_l  :    \mathcal{V}^{  \star }  \to   \mathcal{V}^{ \star }   $  
where    $   F^k_l   \left(   \mathsf{c}_k    \right)    =    \mathsf{c}_ {  \ell }      \,  $. 
 \end{enumerate}

  We start by defining  the family  $  \mathcal{T}  $ whose purpose is to ensure that 
 $  (  \cdot )^{  \star }   :  V^{  \star }     \to        \mathcal{X}   $
 will be a total function: 
 \[
  \mathcal{T}    :=  \Big\{   T_k   :   \    k \in   \mathbb{Z}     \  \mbox{ and  }   
  T_k :=  \big\{  \sigma   \in \mathbb{P}  :   \    \sigma  \left(    \mathsf{c}_k  \right)    \mbox{ is defined } \big\}
  \    \Big\} 
  \        .
 \]
Each   $  T_k    $  is dense:  given  $  \sigma  \in \mathbb{P} $, 
either  $  \sigma  \in T_k  $  or we can extend $  \sigma  $  to   $  \tau  \in  T_k  $  by 
choosing  $ w   \in   V^{  \star }   \setminus   \cup  \sigma  \left[  \mathsf{dom} \left(   \sigma  \right)  \right]  $
and   setting      
\[
  \tau  \left(    \mathsf{c}_k  \right) :=  
   \lbrace   \mathsf{c}_i  :   \   i  \in  \omega   \,   \rbrace  \cup     \lbrace   w  ,    \mathsf{c}_k     \    \rbrace 
 \       .
 \]  
We can do this since   $  V^{  \star }   \setminus   \cup   \sigma  \left[  \mathsf{dom} \left(   \sigma  \right)  \right]    $   is infinite by how  $  \mathbb{P} $  is defined. 
We need to check that  $  \tau  \in  \mathbb{P} $. 
Clearly,  $  \tau $  is one-to-one since  $  \sigma $ is one-to-one and   $    \lbrace   \mathsf{c}_i  :   \   i  \in  \omega   \,   \rbrace  \cup     \lbrace   w  ,    \mathsf{c}_k     \    \rbrace   $  is not in the image of $  \sigma $  by how $w $  is chosen. 
We also have  $  u \in  \tau (u ) $  for all  $  u \in \mathsf{dom} \left(  \tau \right) $  by how  $  \tau $  is defined and the fact that $  \sigma  \in  \mathbb{P} $. 
Finally, since any two sets in $  \mathcal{X} $  differ by finitely many elements, 
there exist two finite sets  $  X , Y  \subseteq   V^{  \star }  $  such that  
\[
\cap   \tau   \left[   \mathsf{dom} \left(  \tau \right)   \right]   =    
  \cap   \sigma   \left[   \mathsf{dom} \left(  \sigma \right)   \right]   \setminus    X  
\       \     \mbox{  and   }      \     \  
\cup   \tau   \left[   \mathsf{dom} \left(  \tau \right)   \right]    =    
\cup   \sigma   \left[   \mathsf{dom} \left(  \sigma \right)   \right]   \cup   Y    
\      .
\]
Since  $  \sigma   \in  \mathbb{P} $ and 
 $   \mathsf{dom} \left(  \tau \right)   =   \mathsf{dom} \left(  \sigma \right)    \cup  \lbrace   \mathsf{c}_k  \rbrace   $, 
it follows that  
$  \cap   \tau   \left[   \mathsf{dom} \left(  \tau \right)   \right]   \setminus   \mathsf{dom} \left(  \tau \right)   $  
and  $  V^{  \star }  \setminus  \cup   \tau   \left[   \mathsf{dom} \left(  \tau \right)   \right]    $  are both infinite. 
This completes the verification that $  \tau  \in  \mathbb{P} $.

Next, we define the family    $  \mathcal{S}  $ whose purpose is to ensure that 
 $  (  \cdot )^{  \star }   $
 will be a surjective function: 
  \[
  \mathcal{S}    :=  \Big\{   S_X   :   \   X \in   \mathcal{X}      \  \mbox{ and  }   
  S_X :=  \big\{  \sigma   \in \mathbb{P}  :   \   
   \left(  \exists  k  \in  \mathbb{Z}  \right)  \left[       \sigma  \left(    \mathsf{c}_k  \right)     =  X   \right]      \    \big\}
  \    \Big\} 
  \        .
 \]
Each       $  S_X   $ is dense: 
 given  $  \sigma  \in \mathbb{P} $, 
either  $  \sigma  \in   S_X  $  or we can extend $  \sigma  $  to   $  \tau  \in  S_X  $  by  
choosing 
\[
x  \in  X   \cap     \left(    \cap    \sigma  \left[  \mathsf{dom} \left(   \sigma  \right)  \right]     \setminus     \mathsf{dom} \left(   \sigma  \right)    \right)  
\]
and setting      $  \tau  \left(    x  \right) :=  X   $.
We  can do this since   $   \cap    \sigma  \left[  \mathsf{dom} \left(   \sigma  \right)  \right]     \setminus     \mathsf{dom} \left(   \sigma  \right)   $   is infinite,  by how  $  \mathbb{P} $  is defined, 
and  any two sets  in  $   \mathcal{X} $  differ by finitely many elements. 
Clearly,  $  \tau $  is one-to-one and  $  u  \in  \tau (u)  $  for all  $  u   \in   \mathsf{dom} \left(   \tau  \right)    $  
since  $  x  \in  X  $  and  $  \sigma   \in  \mathbb{P} $. 
Since  any two sets in $  \mathcal{X} $  differ by finitely many elements,  
$  \cap   \tau   \left[   \mathsf{dom} \left(  \tau \right)   \right]   \setminus   \mathsf{dom} \left(  \tau \right)   $  
and  $  V^{  \star }  \setminus  \cup   \tau   \left[   \mathsf{dom} \left(  \tau \right)   \right]    $  are both infinite
since  $  \sigma  \in  \mathbb{P} $.

We thus have the following lemma.

\begin{lemma}  \label{SecondBasicLemma}
$   \mathcal{T}  $  and   $   \mathcal{S}   $  are  countable families of dense subsets of  $  \mathbb{P} $. 
Furthermore, if  $   \mathcal{D}   $  is a family of dense subsets of $  \mathbb{P} $ with  $ \mathcal{T}  \cup   \mathcal{S}    \subseteq  \mathcal{D} $, 
and  $G$ is a  $   \mathcal{D}   $-generic ideal, 
then $  \cup G   :  V^{  \star }     \to        \mathcal{X}   $ is a bijection.    
\end{lemma}

We proceed to define   $  \mathcal{G} $. 
   Recall that  the purpose of $   \mathcal{G}   $   is to ensure that  for  $k, \ell   \in   \mathbb{Z} $, 
   we can use a  back-and-forth argument  to construct 
an  automorphism  $  F   :    \mathcal{V}^{  \star }  \to   \mathcal{V}^{ \star }   $  
such that   $   F    \left(   \mathsf{c}_k    \right)    =    \mathsf{c}_ {  \ell }      \,  $. 
 This is the motivation for the following definitions.

 \begin{definition}
  Let $ x_0,  x_1,  \ldots , x_m$ be elements of  $   V^{  \star }  $. 
 Let  $  \sigma   \in   \mathbb{P} $. 
 We define the set  $  \Delta \left(   \sigma  \, ;  \,   x_0,  x_1,  \ldots , x_n  \right)    $ to be the smallest set  $  W   \subseteq V^{  \star }    $  such that: 
 \begin{enumerate}
\item       $   \lbrace x_0,  x_1,  \ldots , x_n  \rbrace    \subseteq  W  $
 
 \item  $   \sigma \left( u  \right)   \setminus   \sigma \left( v  \right)       \subseteq   W    $  
     for all    $ u, v   \in  W  \cap   \mathsf{dom} \left(  \sigma  \right)  $

 \item   If  $  v  \in   \mathsf{dom} \left(   \sigma \right)  $  is such that  
 $  \sigma (v)  =  \left(   \sigma (u)  \setminus  A  \right)  \cup B  $  for some  $ u \in  W  $  and finite sets $  A,  B  \subseteq   W  $,  
 then  $  v  \in  W  $. 
 \end{enumerate}
 We will say that  the sequence     $ x_0,  x_1,  \ldots , x_m$    is  $  \mathsf{Good}_1  \left(  \sigma  \right)  $  if: 
 \begin{enumerate}
\item  $  \sigma $  is defined on all of    $   \Delta \left(   \sigma  \, ;  \,   \vec{x}  \right)    $.

\item     For all   $  u  \in    \Delta \left(   \sigma  \, ;  \,   \vec{x}  \right)  $  and all  finite sets 
 $  A, B  \subseteq   \Delta \left(   \sigma  \, ;  \,   \vec{x}  \right)    $, 
  there exists  $  v \in    \Delta \left(   \sigma  \, ;  \,   \vec{x}  \right)    $  such that  
$  \sigma (v)  =  \left(   \sigma (u )  \setminus  A   \right)   \cup   B  $. 
 \end{enumerate}
 \end{definition}

 \begin{definition}
Let     $  \sigma    \in   \mathbb{P} $. 
The  binary  relation  $  \in^{   \sigma  }  $  on  $   V^{  \star }   $   is defined  as follows:  for  $  i ,  j  \in  \mathbb{Z} $ 
 \[
    \mathsf{c}_i   \in^{  \sigma  }   \mathsf{c}_j  
    \     \     \Leftrightarrow_{ \mathsf{def}  }     \     \  
   \mathsf{c}_j   \in   \mathsf{dom} \left(   \sigma   \right)   \;  \wedge   \;    \mathsf{c}_i   \in      \sigma  \left(  \mathsf{c}_j   \right)        
     \       .
     \]
 \end{definition}

\begin{definition}
Let     $   \sigma    \in   \mathbb{P} $. 
We  will  say that  the (possibly empty)  sequence   $  \left( x_0, y_0  \right)   $,   $\left( x_1, y_1  \right) $,     $  \ldots    $,    $\left( x_m, y_m  \right) $  of pairs of elements of $  V^{  \star }   $    is   $ \mathsf{Good}_2  \left(   \sigma    \right)   $ if: 
\begin{enumerate}
\item  $ x_0, x_1 ,  \ldots , x_m $  is   $ \mathsf{Good}_1  \left(   \sigma    \right)   $, 
and   $ y_0, y_1 ,  \ldots , y_m $  is   $ \mathsf{Good}_1  \left(   \sigma    \right)   $.

\item  There exists an isomorphism    
     \[
f :       \left(  \Delta  \left(  \sigma  \,  ;  \,     \vec{x}  \right)    \,  ,   \,  \in^{   \sigma   }    \right)       \to   
      \left(   \Delta      \left(   \sigma     \,  ;  \,    \vec{y}     \right)   \,  ,   \,  \in^{   \sigma   }    \right)  
      \]
  such that        $  f (  x_i ) =  y_i  $ for  all   $ i  \in  \lbrace 0, 1,  \ldots , m  \rbrace $. 
\end{enumerate}
\end{definition}

We will abuse notation and also write  $  \mathsf{Good}_2  \left(   \sigma  \right)  $  for the set  
\[
\Big\{  
 s   \in  \left(  V^{ \star  }   \times   V^{ \star  }    \right)^{  <  \omega  }    :    \    s   \mbox{  is   }     \mathsf{Good}_2  \left(   \sigma  \right)  
   \,  
 \Big\}
\     .
\]
Similarly,  we will also  write     $  \mathsf{Good}_1  \left(   \sigma  \right)  $ for the extension of  
  $  \mathsf{Good}_1  \left(   \sigma  \right)   \,   $.

Before we define  $  \mathcal{G} $, we prove  some properties of   $  \mathsf{Good}_1 $   and   $  \mathsf{Good}_2 $ that will be used in the proof of  Theorem  \ref{maintheorem}.

 \begin{lemma}    \label{ZerothLemmaGood}
  Let  $  \sigma  \in   \mathbb{P} $.
Let   $  y  \in V^{  \star }   $. 
There  exists     $  \tau  \in  \mathbb{P}  $  such that 
$  y  \in  \mathsf{dom} \left(   \tau   \right)   $,  
  $  \sigma  \subseteq    \tau  $   and any finite sequence  in  $  \mathsf{dom} \left(  \tau  \right)  $  is    $   \mathsf{Good}_1 \left(   \tau   \right)     \,    $.
 \end{lemma}
  \begin{proof}

  By how $  \mathbb{P} $ is defined,  
  $   \cap    \sigma  \left[   \mathsf{dom} \left(  \sigma  \right)    \right]     \setminus \mathsf{dom} \left(  \sigma  \right)     $
 and    $ V^{  \star }  \setminus   \cup    \sigma  \left[   \mathsf{dom} \left(  \sigma  \right)    \right]    $
 are both infinite. 
 There thus exist  two infinite sets   $  C,   D   \subseteq  V^{  \star }  $  such that  
 \begin{enumerate}
\item   $  C  \subseteq    \cap    \sigma  \left[   \mathsf{dom} \left(  \sigma  \right)    \right]    \setminus \mathsf{dom} \left(  \sigma  \right)       $  
and  
$   \cap    \sigma  \left[   \mathsf{dom} \left(  \sigma  \right)    \right]      \setminus   \left(   C  \cup  \mathsf{dom} \left(  \sigma  \right)    \right)     $ 
is infinite.

 \item   $  D  \subseteq    V^{  \star }  \setminus   \cup    \sigma  \left[   \mathsf{dom} \left(  \sigma  \right)    \right]       $  
and  $   V^{  \star }  \setminus   \left(  D  \cup   \cup    \sigma  \left[   \mathsf{dom} \left(  \sigma  \right)    \right]     \right)   $ 
is infinite. 
 \end{enumerate}
 We assume also that 
 \[
 y  \in   \mathsf{dom} \left(  \sigma  \right)   \cup  C   \cup   D
 \     .
 \]
Fix  partitions of  $  C$  and  $  D  $  into infinitely many disjoint  infinite sets 
\[
C  =   \bigcup_{  i  \in  \omega  }   C_i   
\     \       \mbox{   and    }      \      \    
D  =   \bigcup_{  i  \in  \omega  }   D_i   
\       .
\] 
 We assume
 \[
 y  \in   \mathsf{dom} \left(  \sigma  \right)   \cup  C_0   \cup   D_0
 \     .
 \]

Let  $  E  \subseteq  V^{  \star }  $  denote the set of all $  x  \in   V^{  \star }  $  such that  
$  x  \in  \sigma (u)  \setminus  \sigma (v)  $  for some  $  u,v    \in       \mathsf{dom} \left(  \sigma  \right)    $. 
We clearly have 
\[
E  \cap     \left(    \cap    \sigma  \left[   \mathsf{dom} \left(  \sigma  \right)    \right]    \right)     =  \emptyset  
\     \     \mbox{  and   }     \      \   
E  \subseteq        \cup    \sigma  \left[   \mathsf{dom} \left(  \sigma  \right)    \right]  
\           .
\]

We construct an increasing sequence   $   \langle   \tau_n  :   \   n  \in  \omega   \,  \rangle  $  of partial one-to-one maps  
$  V^{  \star }  \to   \mathcal{X} $  such that: 
\begin{enumerate}
\item   $  \mathsf{dom} \left(  \sigma   \right)   \cup    \lbrace y  \rbrace  =     \mathsf{dom} \left(  \tau_0   \right)   $  
and  $  \tau_0 $  is an extension of  $  \sigma  $.

\item    $    \mathsf{dom} \left(  \tau_{n}    \right)     \subseteq   
 \mathsf{dom} \left(  \sigma    \right)    \cup  E   \cup   \bigcup_{ i = 0 }^{ n }   C_i  \cup  D_i  $  
 for all $  n  \in  \omega        \,    $.

\item   $  u  \in  \tau_n (u)  $  for all $  n   \in  \omega $  and all $  u  \in   \mathsf{dom} \left(  \tau_n  \right)      \,     $.

\item  $  \cup    \tau_n  \left[   \mathsf{dom} \left(  \tau_n  \right)   \right]  
 \subseteq   
 \cup    \sigma  \left[   \mathsf{dom} \left(  \sigma  \right)   \right]  
 \cup   \bigcup_{ i \leq   n  }    D_i   
 $
 for all  $  n  \in  \omega      \,    $.

 \item  $  \cap    \tau_n  \left[   \mathsf{dom} \left(  \tau_n  \right)   \right]  
  \supseteq   
 \cap    \sigma  \left[   \mathsf{dom} \left(  \sigma  \right)   \right]  
 \setminus     \bigcup_{ i \leq   n  }    C_i     
 $
 for all  $  n  \in  \omega      \,    $.

\item   $   \tau_n (u)   \setminus  \tau_n (v)    \subseteq    \mathsf{dom} \left(  \tau_{n+1}   \right)          $
for  all   $  n  \in  \omega  $   and    all $ u, v    \in   \mathsf{dom} \left(  \tau_n   \right)        \,    $.

\item For all $  n  \in  \omega  $,   all $ u  \in   \mathsf{dom} \left(  \tau_n   \right)  $  and all finite sets   $  A, B   \subseteq   \mathsf{dom} \left(  \tau_n   \right)   $, 
there exists  $  v  \in  \mathsf{dom} \left(  \tau_{n+1}   \right)    $  such that 
$  \tau_{ n+1}  \left(  v  \right)  =   \left(    \tau_n (u)   \setminus  A  \right)   \cup   B         \,     $.

\end{enumerate}
We can then set  $  \tau :=   \bigcup_{  n  \in   \omega  }   \tau_n $   and we will be done.

We proceed to construct the sequence    $   \langle   \tau_n  :   \   n  \in  \omega   \,  \rangle        \,     $.
We start by defining  $  \tau_0  $. 
 If  $  y  \in      \mathsf{dom} \left(  \sigma   \right)  $, let   $  \tau_0 :=  \sigma  $. 
 Assume  $  y  \not\in     \mathsf{dom} \left(  \sigma   \right)   $. 
 Pick  $  u_y   \in     \mathsf{dom} \left(  \sigma   \right) $  and     $  z_y  \in D_0 $  and define  
 $  \tau_0  :    \mathsf{dom} \left(  \sigma   \right)   \cup    \lbrace y  \rbrace    \to    \mathcal{X}  $  as follows:  
 \[
 \tau_0  ( x)   =  \begin{cases}
 \sigma (x)     &    \mbox{  if  }  x  \in     \mathsf{dom} \left(  \sigma   \right)  
 \\
 \sigma ( u_y )  \cup   \lbrace  z_y  ,  y   \rbrace         &   \mbox{  if  }   x =  y  
 \       .
 \end{cases}
 \]
 The map $  \tau_0 $  is one-to-one  since  $  \sigma $  is one-to-one  
 and  $   z_y   \not \in  \sigma (z) $  for all   $   z  \in     \mathsf{dom} \left(  \sigma   \right)          $
 by how  $  D  $  was chosen.

 Assume we have defined  $   \langle   \tau_j   :   \     j  \leq   n  \,  \rangle        \,     $.
 We define  $  \tau_{ n+1}   $. 
 Fix  $  w^{  \star }   \in   \mathsf{dom} \left(  \tau_n   \right)  $. 
 We extend  $  \tau_n  $  to  $  \tau_{ n+1}  $  as follows: 
 \begin{enumerate}
\item  For each  $  x   \in  V^{  \star }   \setminus     \mathsf{dom} \left(  \tau_n   \right)    $   such that  
$  x  \in  \tau_n (u)  \setminus  \tau_n (v)  $  for some  $  u, v  \in      \mathsf{dom} \left(  \tau_n   \right)    $, 
pick a unique  $  z_x  \in   D_{ n+1}    $  and set  
\[
\tau_{ n+1}  \left(   x  \right)   :=     \tau_n  \left(     w^{  \star }    \right)   \cup  \lbrace  z_x  ,  x   \rbrace   
\        .
\]

\item   For  all  $  X  \not\in   \tau_n  \left[      \mathsf{dom} \left(  \tau_n   \right)   \right]  $  such that  
$  X  =   \left(   \tau_n  (u)   \setminus  A   \right)  \cup  B  $  for some  $  u  \in     \mathsf{dom} \left(  \tau_n   \right)   $  and 
finite sets  $  A, B   \subseteq     \mathsf{dom} \left(  \tau_n   \right)   $, 
pick a unique  $  w_X  \in   C_{ n+1}    $  and set  
\[
\tau_{ n+1}  \left(   w_X  \right)   :=   X   =     \left(   \tau_n  (u)   \setminus  A   \right)  \cup  B 
\       .
\] 
 \end{enumerate}
 This completes the definition of  $  \tau_{ n+1}  $. 
 \end{proof}

\begin{lemma}  \label{FirstLemmaGood}
Let  $  \sigma ,  \tau   \in  \mathbb{P} $. 
If  $  \sigma   \subseteq   \tau  $,  then   
\[
     \mathsf{Good}_1  \left(   \sigma  \right)    \subseteq     \mathsf{Good}_1  \left(   \tau  \right)        
     \     \mbox{  and   }     \    
     \mathsf{Good}_2  \left(   \sigma  \right)    \subseteq     \mathsf{Good}_2  \left(   \tau  \right)     
     \               .
\]     
\end{lemma} 
\begin{proof}

If   a sequence  $  \vec{x}  $   is    $  \mathsf{Good}_1  \left(   \sigma  \right)     $, 
then    $  \Delta  \left(  \sigma  \,  ;  \,     \vec{x}  \right)     =  \Delta  \left(  \sigma^{  \prime }   \,  ;  \,     \vec{x}  \right)             $
for all   $  \sigma^{  \prime }   \in   \mathbb{P} $  such that  $  \sigma  \subseteq  \sigma^{  \prime }       \,      $. 
\end{proof}

 \begin{lemma}    \label{SecondLemmaGood}
 Let  $  \sigma  \in  \mathbb{P} $. 
Let  $  x_0, y_0  \in  V^{  \star }  $. 
Assume the one-element sequence $  x_0$  is  $  \mathsf{Good}_1  \left(  \sigma \right) $
and   the one-element sequence $  y_0$  is  $  \mathsf{Good}_1  \left(  \sigma \right) $. 
Then,  $  (x_0,  y_0)  $  is     $  \mathsf{Good}_2  \left(  \sigma \right) $. 
 \end{lemma}
\begin{proof}

Let  $  z  \in  \lbrace  x_0, y_0  \rbrace $. 
Since the one-element sequence $z $  is  $   \mathsf{Good}_1  \left(  \sigma  \right)  $, 
we have  $    \Delta \left(  \sigma   \,  ;   \,   z  \right)    \subseteq     \mathsf{dom} \left(  \sigma  \right)   $  
and for all $  u  \in    \Delta \left(  \sigma   \,  ;   \,   z  \right)    $  and  all finite  sets  $  A , B \subseteq   \Delta \left(  \sigma   \,  ;   \,   z  \right)   $, 
there exists  $  v  \in    \Delta \left(  \sigma   \,  ;   \,   z  \right)    $  such that  
$  \sigma (v)  =   \left(   \sigma (u)  \setminus   A   \right)  \cup   B      \,     $. 
We define the sets  $  \Delta \left(  \sigma , k  \,  ;   \,   z  \right)  $  by recursion on  $ k  \in  \omega   $: 
\begin{enumerate}
\item  $  \Delta \left(  \sigma , 0 \,  ;   \,   z  \right)   =  \lbrace   z  \rbrace $

\item   $  \Delta \left(  \sigma , k +1  \,  ;   \,   z  \right)   =   \Delta \left(  \sigma , k  \,  ;   \,   z  \right)   \cup  W_k  $ 
where  $  W_k $  consists of all $  v  \in  \mathsf{dom} \left(  \sigma  \right)  $  such that 
$  \sigma (v)  =    \sigma (z)  \setminus  A   $  for some  $  A  \subseteq     \Delta \left(  \sigma , k  \,  ;   \,   z \right)        \,    $.
\end{enumerate} 
We   have the following observation.

  \begin{quote} {\bf (Claim)} \;\;\;\;\;\;    
  Let   $   z  \in    \lbrace  x_0,  y_0  \rbrace   $. 
  The following holds for all  $  k  \in   \omega   $: 
  \begin{enumerate}
  \item    $    \Delta \left(  \sigma , k   \,  ;   \,   z  \right)     \subseteq   \sigma (z)       $
  and  $  \sigma (u)  \subseteq   \sigma (z)  $  for all  $  u \in       \Delta \left(  \sigma , k   \,  ;   \,   z  \right)     \,    $.

  \item   $  \sigma (s)  \setminus  \sigma (t)   \subseteq       \Delta \left(  \sigma , k   \,  ;   \,   z  \right)    $
  for all $  s, t  \in     \Delta \left(  \sigma , k   \,  ;   \,   z  \right)         \,   $.

  \item  For all $  u \in   \Delta \left(  \sigma , k   \,  ;   \,   z  \right)    $  and all $  A, B  \subseteq    \Delta \left(  \sigma , k   \,  ;   \,   z  \right)   $,   there exists  $  v  \in    \Delta \left(  \sigma , k+1   \,  ;   \,   z  \right)   $
  such that    $  \sigma (v)  =   \left(   \sigma (u)  \setminus   A   \right)  \cup   B      \,     $.

  \item    $  v   \in  \sigma (w)     $  for    all   $  w  \in    \Delta \left(  \sigma , k   \,  ;   \,   z  \right)    $  and  
all $  v  \in   \Delta \left(  \sigma , k +1  \,  ;   \,   z  \right)     \setminus    \Delta \left(  \sigma , k   \,  ;   \,   z  \right)    $.  
  \end{enumerate}
 \end{quote}

First,  we show that (1) holds for all $ k   \in  \omega   $. 
We prove this by induction on $k$. 
We consider the base case $  k =  0 $. 
This follows from the fact that  $    \Delta \left(  \sigma , 0  \,  ;   \,   z  \right)    =  \lbrace z  \rbrace  $  
and   $  z  \in    \sigma (z)        $  by how    $  \mathbb{P} $  is defined. 
We consider the inductive case   $  k >  0 $. 
Pick  $  u  \in    \Delta \left(  \sigma , k  \,  ;   \,   z  \right)   $. 
We need to show that  $  u  \in  \sigma (z)  $  and  $  \sigma (u)   \subseteq   \sigma (z)  $. 
We have two cases: 
(i)   $  u  \in     \Delta \left(  \sigma , k-1  \,  ;   \,   z  \right)    $; 
(ii) $  u  \not\in   \Delta \left(  \sigma , k-1  \,  ;   \,   z  \right)    $. 
Case  (i) is fine since we just use the induction hypothesis. 
We consider (ii). 
Since $  u  \not\in   \Delta \left(  \sigma , k-1  \,  ;   \,   z  \right)    $,   there exists   $  A  \subseteq    \Delta \left(  \sigma , k-1  \,  ;   \,   z  \right)    $  such that  
$  \sigma (u)  =  \sigma (z)  \setminus  A  $. 
This clearly shows that  $  \sigma (u)   \subseteq    \sigma (z)   $. 
It remains to show that  $  u  \in  \sigma (z)  $. 
But, by how  $  \mathbb{P} $  is defined,  $  u  \in   \sigma (u)   \subseteq   \sigma (z)  $. 
Thus, by induction,  (1)  holds for all $ k  \in  \omega  $.

We show that (2) holds   holds for all $ k   \in  \omega   $. 
We prove this by induction on $k$. 
The base case $k= 0 $  holds since   $   \Delta \left(  \sigma , 0  \,  ;   \,   z  \right)   $  is a singleton. 
We consider the inductive case  $  k >  0  $. 
By how  $   \Delta \left(  \sigma , k  \,  ;   \,   z  \right)   $  is defined,  there exists 
 $  v_0,  v_1,  \ldots ,  v_j  \in   \Delta \left(  \sigma , k-1  \,  ;   \,   z  \right)    $  
 and  $  A_0,  A_1,  \ldots , A_j  \subseteq    \Delta \left(  \sigma , k-1  \,  ;   \,   z  \right)  $  such that  
 \[
 \Delta \left(  \sigma , k  \,  ;   \,   z  \right)  =   \Delta \left(  \sigma , k-1  \,  ;   \,   z  \right)     \cup   
 \left\{
     \sigma^{ -1}   \left(      \sigma ( v_i )  \setminus   A_i   \right)     :    \     i  \in  \lbrace 0, 1,  \ldots  ,  j  \rbrace 
 \right\}
 \      .
 \]
Pick  $  s,  t   \in  \Delta \left(  \sigma , k  \,  ;   \,   z  \right) $.
We need to show that 
 $  \sigma (s)  \setminus  \sigma (t)   \subseteq       \Delta \left(  \sigma , k   \,  ;   \,   z  \right)    $.
 We have two cases:
  (i)  $  s  \in   \Delta \left(  \sigma , k-1  \,  ;   \,   z  \right) $; 
  (ii)  $  s  \not\in   \Delta \left(  \sigma , k-1  \,  ;   \,   z  \right) $. 
  We consider (i). 
  If  $  t  \in   \Delta \left(  \sigma , k-1  \,  ;   \,   z  \right)   $, 
  then   $  \sigma (s)  \setminus  \sigma (t)    \subseteq   \Delta \left(  \sigma , k-1  \,  ;   \,   z  \right)    \subseteq       \Delta \left(  \sigma , k   \,  ;   \,   z  \right)    $
  by the induction hypothesis. 
  Otherwise, $   \sigma (t)  =    \sigma ( v_i )  \setminus   A_i   $  for some  $  i  \in    \lbrace 0, 1,  \ldots  ,  j  \rbrace   $ 
  and hence  by the induction hypothesis 
  \[
    \sigma (s)  \setminus  \sigma (t)        \subseteq   
    \left(       \sigma (s)  \setminus \sigma ( v_i )       \right)    \cup  A_i 
    \subseteq   \Delta \left(  \sigma , k-1  \,  ;   \,   z  \right)    \subseteq       \Delta \left(  \sigma , k   \,  ;   \,   z  \right)  
    \      .
  \]
Finally,  we consider (ii). 
We have  $  \sigma (s)  =    \sigma ( v_{  \ell}   )  \setminus   A_{  \ell  }   $  for some  $  \ell  \in    \lbrace 0, 1,  \ldots  ,  j  \rbrace   $. 
If  $  t   \in    \Delta \left(  \sigma , k-1  \,  ;   \,   z  \right)   $,  then    by the induction hypothesis 
  \[
    \sigma (s)  \setminus  \sigma (t)        \subseteq        \sigma ( v_{  \ell}   )   \setminus \sigma ( t )    
    \subseteq   \Delta \left(  \sigma , k-1  \,  ;   \,   z  \right)    \subseteq       \Delta \left(  \sigma , k   \,  ;   \,   z  \right)  
    \      .
  \]
Otherwise,   $   \sigma (t)  =    \sigma ( v_i )  \setminus   A_i   $  for some  $  i  \in    \lbrace 0, 1,  \ldots  ,  j  \rbrace   $ 
  and hence  by the induction hypothesis 
  \[
    \sigma (s)  \setminus  \sigma (t)        \subseteq       
    \left(   \sigma ( v_{  \ell}   )   \setminus  \sigma ( v_i )      \right)    \cup   A_i 
    \subseteq   \Delta \left(  \sigma , k-1  \,  ;   \,   z  \right)    \subseteq       \Delta \left(  \sigma , k   \,  ;   \,   z  \right)  
    \      .
  \]  
Thus, by induction,  (2)  holds for all $ k  \in  \omega  $.

We show that (3) holds   holds for all $ k   \in  \omega   $. 
Pick     $  u  \in  \Delta \left(  \sigma , k   \,  ;   \,   z  \right)   $  and    $  A,  B   \subseteq  \Delta \left(  \sigma , k   \,  ;   \,   z  \right)   $. 
We need to show that  there exists   $  v  \in    \Delta \left(  \sigma , k +1   \,  ;   \,   z  \right)    $  such that  
$  \sigma (v)  =      \left(   \sigma (u)  \setminus   A   \right)  \cup   B  $. 
By how  $   \Delta \left(  \sigma , k +1   \,  ;   \,   z  \right)      $  is defined,  
  we need to show that 
there exists  $  C  \subseteq    \Delta \left(  \sigma , k   \,  ;   \,   z  \right)   $  such that 
\[
 \left(   \sigma (u)  \setminus   A   \right)  \cup   B    =    \sigma (z)   \setminus  C
 \       .
\]
We have  $ \sigma  (u)  \cup    A  \cup  B   \subseteq   \sigma (z)  $  by clause (1) of the claim.
By clause (2) of the claim,    there exists    $  D  \subseteq      \Delta \left(  \sigma , k   \,  ;   \,   z  \right)   $  such that 
$  \sigma  (u)  =   \sigma (z)  \setminus  D  $. 
Hence 
\begin{align*}
 \left(   \sigma (u)  \setminus   A   \right)  \cup   B    
 &=  
 \left(   \sigma (z)  \setminus   \left(  A  \cup   D  \right)     \right)  \cup   B  
 \\
 &=  \sigma (z)   \setminus  C
&    \mbox{  where  }   C  :=    \left(  A  \cup   D  \right)     \setminus    B  
 \        .
\end{align*}
Thus,   (3)  holds  for all $  k  \in  \omega  $.

Finally,    we   show that (4) holds  for all  $  k  \in  \omega  $.
Pick  $  w  \in    \Delta \left(  \sigma , k   \,  ;   \,   z  \right)    $   and     $  v  \in   \Delta \left(  \sigma , k +1  \,  ;   \,   z  \right)     \setminus    \Delta \left(  \sigma , k   \,  ;   \,   z  \right)    $. 
By clause (2) of the claim 
\[
v  \in   \cap   \sigma   \left[    \Delta \left(  \sigma , k   \,  ;   \,   z  \right)     \right]  
\      \      \mbox{  or   }      \         \   
v    \not\in   \cup   \sigma   \left[    \Delta \left(  \sigma , k   \,  ;   \,   z  \right)     \right]  
\      .
\]
By clause  (1) of the claim,  $  v  \in   \Delta \left(  \sigma , k +1  \,  ;   \,   z  \right)     \subseteq  \sigma (z)  $. 
Hence,  $  v  \in   \cap   \sigma   \left[    \Delta \left(  \sigma , k   \,  ;   \,   z  \right)     \right]    $
since  $  z  \in     \Delta \left(  \sigma , k   \,  ;   \,   z  \right)           \,     $. 
In particular,    $  v   \in  \sigma (w)        \,    $.

It follows clauses (2)-(3) of the  claim   and the closure properties that define    $      \Delta \left(  \sigma   \,  ;   \,   z  \right)   $  that 
\[
  \Delta \left(  \sigma   \,  ;   \,   z  \right)    =  \bigcup_{  k  \in  \omega  }      \Delta \left(  \sigma , k  \,  ;   \,   z  \right)   
  \       . 
\]

We show that    $  (x_0,  y_0)  $  is     $  \mathsf{Good}_2  \left(  \sigma \right) $.
It suffices to construct an increasing sequence  of isomorphisms 
\[
f_k  :    \left(   \Delta \left(  \sigma , k  \,  ;   \,   x_0   \right)    \,  ,  \,   \in^{  \sigma }   \right)  
   \to  
   \left(     \Delta \left(  \sigma , k  \,  ;   \,   y_0  \right)       \,  ,  \,   \in^{  \sigma }   \right)  
\     \    \mbox{  for  }  k  \in  \omega 
\     .
\]
By setting  $  f  :   \bigcup_{  k   \in   \omega }    f_k $,  we then get an isomorphism 
\[
f  :    \left(   \Delta \left(  \sigma   \,  ;   \,   x_0   \right)    \,  ,  \,   \in^{  \sigma }   \right)  
   \to  
   \left(     \Delta \left(  \sigma   \,  ;   \,   y_0  \right)       \,  ,  \,   \in^{  \sigma }   \right)  
\]
that witnesses that    $  (x_0,  y_0)  $  is     $  \mathsf{Good}_2  \left(  \sigma \right)  $.

We construct the sequence    $  \langle f_k  :   \   k  \in  \omega   \,  \rangle  $  by recursion. 
We have  $  f_0  \left(  x_0  \right)   = y_0 $  since 
 $    \Delta \left(  \sigma , 0  \,  ;   \,   z   \right)    =  \lbrace  z  \rbrace  $  for all $  z  \in  \lbrace x_0,  y_0  \rbrace  $. 
This map is an isomorphism since   $  x_0  \in  \sigma \left(  x_0  \right)   $  and    $  y_0  \in   \sigma \left(  y_0  \right)   $.
Assume we have defined  $  \langle f_j  :   \    j  \leq  k    \,  \rangle  $. 
We define  $  f_{ k+1}  $: 
\begin{enumerate}
\item  If   $   v  \in    \Delta \left(  \sigma , k  \,  ;   \,   x_0   \right)    $, then 
$  f_{ k+1}  \left( v  \right)    :=    f_k (v)        \,   $.

\item   If  $  v  \not\in     \Delta \left(  \sigma , k  \,  ;   \,   x_0   \right)     $  
and   hence  $  \sigma ( v )  =   \sigma  ( x_0  )   \setminus  A_v    $
where   $   A_v     \subseteq    \Delta \left(  \sigma , k  \,  ;   \,   x_0   \right)   $, 
then 
\[
 f_{ k+1}  \left( v  \right)     
 := 
 \sigma^{  -1}  \Big(      \sigma  ( y_0 )   \setminus  f_k \left[  A_v  \right]      \Big)  
\      .
\]
\end{enumerate}
First, we show that  $ f_{ k+1}  $  is one-to-one. 
Since  $  f_k$  and $  \sigma $  are both one-to-one,  
it suffices to show that  $  f_{ k+1}  (v)  \not\in    \Delta \left(  \sigma , k  \,  ;   \,   y_0   \right)   $  
for all $  v  \in   \Delta \left(  \sigma , k+1  \,  ;   \,   x_0   \right)   \setminus   \Delta \left(  \sigma , k  \,  ;   \,   x_0   \right)   $.
Assume for the sake of  a contradiction,  there exist
$  w  \in     \Delta \left(  \sigma , k  \,  ;   \,   x_0   \right)   $  and   
$  v  \in   \Delta \left(  \sigma , k+1  \,  ;   \,   x_0   \right)   \setminus   \Delta \left(  \sigma , k  \,  ;   \,   x_0   \right)   $  
such that    
$
f_k (w)  =    f_{ k+1}  \left( v  \right)   
$.
We have 
\[
\sigma \left(   f_k \left( w  \right)     \right)    = 
\sigma \left(   f_{ k+1}  \left( v  \right)     \right)  =      \sigma  ( y_0 )   \setminus  f_k \left[  A_v  \right]   
\       .
\]
Since  $  f_k $  is an isomorphism, we must also have 
\[
\sigma (w)  =     \sigma  ( x_0 )   \setminus   A_v 
\       .
\]
But since  $  \sigma (v)  =   \sigma  ( x_0  )   \setminus  A_v     $  and  $  \sigma  $  is one-to-one, 
we get  that  $  v  =   w  $,  which contradicts the assumption that  
$  v  \not\in     \Delta \left(  \sigma , k  \,  ;   \,   x_0   \right)   $.
Thus,   $  f_{ k+1}  $  is  one-to-one.

It remains to show that that  $f_{ k+1}  $  is an isomorphism. 
Pick    $  v, w  \in  \Delta \left(  \sigma , k+1  \,  ;   \,   x_0   \right)   $. 
We need to show that  $  v  \in   \sigma ( w)   $  if and only if  $   f_{ k+1} ( v )   \in     \sigma  \left(   f_{ k+1} ( w )    \right)     \,    $.
We have two cases: 
(I) $   w  \in  \Delta \left(  \sigma , k  \,  ;   \,   x_0   \right)    $; 
(II)   $   w    \not\in  \Delta \left(  \sigma , k  \,  ;   \,   x_0   \right)    $  and 
$  \sigma ( w )  =   \sigma  ( x_0  )   \setminus  A    $
where    $   A    \subseteq    \Delta \left(  \sigma , k  \,  ;   \,   x_0   \right)   $.
We consider case (I). 
If  $  v  \in  \Delta \left(  \sigma , k \,  ;   \,   x_0   \right)   $,  
then   $  v  \in   \sigma ( w)   $  if and only if  $   f_{ k+1} ( v )   \in     \sigma  \left(   f_{ k+1} ( w )    \right)        $
since  $ f_k $  is an isomorphism and  $ f_{ k+1} $  agrees with $  f_k  $  on   $  \Delta \left(  \sigma , k \,  ;   \,   x_0   \right)   $. 
Assume   $  v    \not\in  \Delta \left(  \sigma , k \,  ;   \,   x_0   \right)   $. 
Since  $  f_{ k+1}  $  is one-to-one,   we also have  $  f_{ k+1}  \left(  v  \right)     \not\in  \Delta \left(  \sigma , k \,  ;   \,   y_0   \right)   $. 
By clause (4) of the claim
\[
v  \in   \sigma ( w)     \     \mbox{  and  }        \   
 f_{ k+1} ( v )   \in    \sigma  \left(   f_{ k} ( w )    \right)     =  \sigma  \left(   f_{ k+1} ( w )    \right)      
 \            .
\]

We consider Case (II):   $  \sigma ( w )  =   \sigma  ( x_0  )   \setminus  A    $
and  $  \sigma   \left(   f_{ k+1}    (w )    \right)   =   \sigma  (\left(   y_0    \right)   \setminus   f_k  \left[ A  \right]      $ 
where    $   A    \subseteq    \Delta \left(  \sigma , k  \,  ;   \,   x_0   \right)   $.
We have two subcases: 
(IIa)  $  v  \in    \Delta \left(  \sigma , k \,  ;   \,   x_0   \right)    $; 
(IIb)  $  v  \not\in   \Delta \left(  \sigma , k \,  ;   \,   x_0   \right)       \,    $. 
We consider (IIa). 
Since  $f_k $  is an isomorphism and  $  f_{ k+1}  $  agrees with  $  f_k  $  on   $   \Delta \left(  \sigma , k  \,  ;   \,   x_0   \right)   $
\begin{align*}
 v  \in   \sigma ( w)       \   
 &  \Leftrightarrow     \  
 v   \in  \sigma ( x_0 )     \;   \wedge    \;  v  \not\in  A  
 \\
  &  \Leftrightarrow     \  
 f_k  (v)    \in  \sigma  \left(   y_0  \right)     \;   \wedge    \;    f_k  (v)     \not\in  f_k  \left[  A  \right] 
   \\
  &  \Leftrightarrow     \  
 f_{k+1}  (v)    \in  \sigma  \left(   y_0  \right)     \;   \wedge    \;    f_k  (v)     \not\in  f_k  \left[  A  \right] 
       \\
  &  \Leftrightarrow     \  
   f_{k+1}  (v)     \in    \sigma   \left(    f_{k+1}  (w)    \right)  
   \        .
\end{align*}

Finally,  we consider (IIb). 
Since   $  v  \not\in   \Delta \left(  \sigma , k \,  ;   \,   x_0   \right)    $,  
we also have     $  f_{ k+1}   \left( v  \right)    \not\in   \Delta \left(  \sigma , k \,  ;   \,   y_0   \right)    $. 
By clause (4)  of the claim  
\[
v  \in  \cap   \sigma  \left[       \Delta \left(  \sigma , k \,  ;   \,   x_0   \right)     \right]   
\    \     \mbox{  and   }     \      \  
f_{ k+1}  \left( v  \right)   \in  \cap   \sigma  \left[       \Delta \left(  \sigma , k \,  ;   \,   y_0   \right)     \right]   
\     .
\] 
Hence,  since   $  A  \subseteq      \Delta \left(  \sigma , k \,  ;   \,   x_0   \right)    $
\[
v  \in     \sigma  ( x_0  )   \setminus  A   =   \sigma ( w )  
\     \    \mbox{  and   }    \     \
f_{ k+1}  \left( v  \right)   \in        \sigma  (\left(   y_0      \right)   \setminus   f_k  \left[ A  \right]        =   \sigma   \left(   f_{ k+1}    (w )    \right)  
\      .
\]
This completes the proof that  $  f_{ k+1}  $  is an isomorphism.
\end{proof}

 We are ready to complete the definition of  $  \mathcal{D} $. 
 Recall that  $   \mathcal{D}   :=   \mathcal{T}  \cup   \mathcal{S}  \cup  \mathcal{G}   $ and it remains to define  $  \mathcal{G}  $. 
 For each    sequence   $  (a_0,  b_0 ) ,  \ldots , (a_m ,  b_m ) $ of pairs of elements  of    $  V^{  \star }   $
 and each $    a_{ m+1}   \in  V^{  \star } $, 
 let 
 \[
 \mathcal{G}   \left(    \,      (a_0,  b_0 ) ,  \ldots , (a_m , b_m ) \,  ;   \,   a_{ m+1}      \,    \right) 
\]  
consist  of all conditions  $  \sigma  \in  \mathbb{P} $ for which  one of the following holds: 
   \begin{enumerate}
   
\item   There exists     $  b_{  m+1 }   \in   V^{  \star }  $ such that   
$  (a_0,  b_0 ) $,  $\ldots $,   $(a_m , b_m )  $,      $ (a_{ m+1 } ,  b_{ m+1 } ) $  is  $ \mathsf{Good}_2 \left(   \sigma  \right)   $.

\item There is no   $  \mathbb{P}   \ni  \tau \supseteq  \sigma $  such that 
  $  (a_0,  b_0 )   $,   $\ldots   $,  $ (a_m , b_m ) $  is  $ \mathsf{Good}_2  \left(   \tau   \right)   $.
 \end{enumerate} 
 Let    $ \mathcal{G}  $ denote the family of all  the    
 $   \mathcal{G}    \left(     (a_0,  b_0 ) ,  \ldots , (a_m , b_m ) \,  ;   \,   a_{ m+1}      \,    \right)  $. 
The next   lemma  shows that $  \mathcal{G}  $ consists of dense subsets of  $  \mathbb{P}  $.

\begin{lemma}  \label{SectionWDExtensionLemma}
Let  $  \sigma  \in  \mathbb{P} $. 
Assume   $  \left( a_0, b_0  \right),    \ldots , \left( a_m, b_m  \right) $     is  $ \mathsf{Good}_2 \left(   \sigma  \right)    $. 
Let    $  a_{ m+1 }   \in  V^{  \star }  $ be such that   $ a_0,  \ldots , a_m , a_{ m+1 } $  is   $ \mathsf{Good}_1 \left(   \sigma   \right)   $. 
Then, there exist  $  \tau \in  \mathbb{P}  $  and  $  b_{ m+1}  \in  V^{  \star }  $  such that  
\begin{enumerate}
\item   $  \sigma  \subseteq  \tau  $

\item  the  sequence 
 $  \left( a_0, b_0  \right) $, $     \ldots  $,    $  \left( a_m, b_m  \right)   $,      $    \left( a_{ m+1 } , b_{ m+1}   \right) $    
  is  $  \mathsf{Good}_2  \left(   \tau  \right)   $. 
\end{enumerate}
\end{lemma}
\begin{proof}
 
  By how $  \mathbb{P} $ is defined,  
  $   \cap    \sigma  \left[   \mathsf{dom} \left(  \sigma  \right)    \right]    \setminus  \mathsf{dom} \left(  \sigma  \right)  $
 and    $ V^{  \star }  \setminus   \cup    \sigma  \left[   \mathsf{dom} \left(  \sigma  \right)    \right]    $
 are both infinite. 
 Choose two infinite sets   $  C,   D   \subseteq  V^{  \star }  $  such that  
 \begin{enumerate}
\item   $  C  \subseteq   \cap    \sigma  \left[   \mathsf{dom} \left(  \sigma  \right)    \right]     \setminus    \mathsf{dom} \left(  \sigma  \right)   $  
and 
 $   \cap    \sigma  \left[   \mathsf{dom} \left(  \sigma  \right)    \right]      \setminus       \left(  C  \cup    \mathsf{dom} \left(  \sigma  \right)    \right) $ 
is infinite.

 \item   $  D  \subseteq    V^{  \star }  \setminus   \cup    \sigma  \left[   \mathsf{dom} \left(  \sigma  \right)    \right]       $  
and  $   V^{  \star }  \setminus   \left(  D  \cup   \cup    \sigma  \left[   \mathsf{dom} \left(  \sigma  \right)    \right]     \right)   $ 
is infinite. 
 \end{enumerate}
Since     $  \left( a_0, b_0  \right),    \ldots , \left( a_m, b_m  \right) $     is  $ \mathsf{Good}_2 \left(   \sigma  \right)    $, 
there exists an isomorphism 
    \[
f :       \left(  \Delta  \left(  \sigma  \,  ;  \,     \vec{a}  \right)    \,  ,   \,  \in^{   \sigma   }    \right)       \to   
      \left(   \Delta      \left(   \sigma     \,  ;  \,    \vec{b}     \right)   \,  ,   \,  \in^{   \sigma   }    \right)  
      \]
  such that        $  f (  a_i ) =  b_i  $ for  all   $ i  \in  \lbrace 0, 1,  \ldots , m  \rbrace $. 
We choose   a one-to-one map 
\[
g :     \Delta  \left(  \sigma  \,  ;  \,     a_0, a_1,  \ldots , a_{ m+1}   \right)     \to   V^{  \star }  
\]
such that: 
\begin{enumerate}
\item  $g$ is an extension of   $  f $.

\item For  $  x  \in   \Delta  \left(  \sigma  \,  ;  \,     a_0, a_1,  \ldots , a_{ m+1}   \right)    \setminus  
 \Delta  \left(  \sigma  \,  ;  \,     a_0, a_1,  \ldots , a_{ m}   \right)   $
\[
x   \in  \cap   \sigma  \left[   \Delta  \left(  \sigma  \,  ;  \,     a_0, a_1,  \ldots , a_{ m}   \right)    \right]  
\     \       \Leftrightarrow        \         \   
g(x)    \in   C
\        .
\]

\item For  $  x  \in   \Delta  \left(  \sigma  \,  ;  \,     a_0, a_1,  \ldots , a_{ m+1}   \right)    \setminus  
 \Delta  \left(  \sigma  \,  ;  \,     a_0, a_1,  \ldots , a_{ m}   \right)   $
\[
x    \not\in  \cup   \sigma  \left[   \Delta  \left(  \sigma  \,  ;  \,     a_0, a_1,  \ldots , a_{ m}   \right)    \right]  
\     \       \Leftrightarrow        \         \   
g(x)    \in   D
\        .
\]
\end{enumerate}

We define  $  \tau  \in  \mathbb{P} $  such that  
$  \sigma   \subseteq   \tau  $  and  $g$ is an isomorphism 
    \[
g :       \left(  \Delta  \left(  \tau  \,  ;  \,     a_0,  \ldots , a_m , a_{ m+1}  \right)    \,  ,   \,  \in^{   \tau   }    \right)     
  \to   
      \left(   \Delta      \left(   \tau     \,  ;  \,      b_0,  \ldots , g_m ,   g  \left( a_{ m+1}  \right)    \right)   \,  ,   \,  \in^{   \tau   }    \right) 
  \       .     
      \]
We use the fact that any two sets in  $  \mathcal{X} $  differ by finitely  many elements.    
Fix  $  w^{  \star }   \in       \Delta  \left(  \sigma  \,  ;  \,     a_0,  \ldots , a_m   \right)       \,   $. 
Pick  $   v  \in  \Delta  \left(  \sigma  \,  ;  \,     a_0,  \ldots , a_m , a_{ m+1}  \right)     \setminus  
\Delta  \left(  \sigma  \,  ;  \,     a_0,  \ldots , a_m  \right)        \,     $. 
There exists a unique pair of finite sets  $  A_v ,  B_v  \subseteq  V^{  \star }     $   such that: 
\begin{enumerate}
\item   $  A_v  \subseteq    \sigma  \left(   w^{  \star }  \right)  $  

\item   $  B_v  \cap    \sigma  \left(   w^{  \star }  \right)   =   \emptyset  $

\item   $   \sigma (v)  =   \left(    \sigma  \left(   w^{  \star }  \right)   \setminus  A_v  \right)   \cup   B_v        \,  $. 
\end{enumerate}
We have 
\[
 A_v   \cup   B_v  \subseteq   \Delta  \left(  \sigma  \,  ;  \,     a_0,  \ldots , a_m , a_{ m+1}  \right)  
 \]
 since  $  \sigma (s )  \setminus  \sigma (t)   \subseteq    \Delta  \left(  \sigma  \,  ;  \,     a_0,  \ldots , a_m , a_{ m+1}  \right)    $  
 for all $  s, t  \in   \Delta  \left(  \sigma  \,  ;  \,     a_0,  \ldots , a_m , a_{ m+1}  \right)    $. 
We have  
\[
 A_v  \setminus    \Delta  \left(  \sigma  \,  ;  \,     a_0,  \ldots , a_m  \right)     \neq   \emptyset   
 \      \      \vee       \        \  
 B_v  \setminus    \Delta  \left(  \sigma  \,  ;  \,     a_0,  \ldots , a_m   \right)    \neq   \emptyset 
 \tag{*}
\]
since, by how   $  \Delta  \left(  \sigma  \,  ;  \,     a_0,  \ldots , a_m   \right)     $   is defined, 
   for all  finite sets   $  Y_0,  Y_1    \subseteq   \Delta  \left(  \sigma  \,  ;  \,     a_0,  \ldots , a_m   \right)     $  there exists  
$  u  \in   \Delta  \left(  \sigma  \,  ;  \,     a_0,  \ldots , a_m   \right)     $  such that  
$  \sigma (u)  =   \left(    \sigma  \left(   w^{  \star }  \right)   \setminus  Y_0   \right)   \cup   Y_1        \,  $. 
Let  
\[
\tau  \left(   g(v) \right)   :=   
 \left(    \sigma  \left(   g \left( w^{  \star }   \right)   \right)   \setminus   g  \left[ A_v   \right]  \right)   \cup    g  \left[ B_v   \right]  
 \          .
\]
Since  $  \sigma $ and $g$   are  one-to-one, 
it follows from (*) and the uniqueness of $ (A_v, B_v ) $  that $  \tau $  is one-to-one. 
By  how  $  \tau $  is defined, we have 
\begin{enumerate}
\item   $  \cap   \tau  \left[   \mathsf{dom} \left(  \tau  \right)   \right]    \supseteq 
\left(  \cap   \sigma  \left[   \mathsf{dom} \left(  \sigma  \right)   \right]     \setminus  C   \right)  $

\item     $  \cup   \tau  \left[   \mathsf{dom} \left(  \tau  \right)   \right]    \subseteq 
\cup   \sigma  \left[   \mathsf{dom} \left(  \sigma  \right)   \right]     \cup   D     \,      $.
\end{enumerate}
It follows that   $  \tau  \in   \mathbb{P} $ by how $C$ and $D$ were chosen.

We need to check that  for all  $  u,  z  \in   \Delta  \left(  \sigma  \,  ;  \,     a_0,  \ldots , a_m , a_{ m+1 }  \right)     $
\[
u  \in  \tau ( z )  
\    \Leftrightarrow    \   
g(u)  \in   \tau \left(  g(z)   \right)  
\       .     \tag{**}
\]
We have two cases: 
(I)  $  z  \in   \Delta  \left(  \sigma  \,  ;  \,     a_0,  \ldots , a_m   \right)     $; 
(II)  $  z    \not\in   \Delta  \left(  \sigma  \,  ;  \,     a_0,  \ldots , a_m   \right)     $.
We consider (I). 
If  $  u \in   \Delta  \left(  \sigma  \,  ;  \,     a_0,  \ldots , a_m   \right)   $, 
then (**) holds since  $f$  is an isomorphism and $g$  agrees with $f$  on  $   \Delta  \left(  \sigma  \,  ;  \,     a_0,  \ldots , a_m   \right)     $. 
Assume  $  u   \not\in   \Delta  \left(  \sigma  \,  ;  \,     a_0,  \ldots , a_m   \right)   $. 
Then, by how  $g$  is defined and the fact that  $  \sigma (s)  \setminus  \sigma (t)   \subseteq    \Delta  \left(  \sigma  \,  ;  \,     a_0,  \ldots , a_m   \right)   $  for all $  s, t  \in   \Delta  \left(  \sigma  \,  ;  \,     a_0,  \ldots , a_m   \right)    $, 
we have the following cases: 
\begin{enumerate}
\item  $  u  \in     \cap   \sigma   \left[  \Delta  \left(  \sigma  \,  ;  \,     a_0,  \ldots , a_m   \right)     \right]  $  
and  $  g(u)  \in  C    \subseteq    \cap   \sigma   \left[  \Delta  \left(  \sigma  \,  ;  \,     b_0,  \ldots , b_m   \right)     \right]  $

\item     $  u   \not\in     \cup   \sigma   \left[  \Delta  \left(  \sigma  \,  ;  \,     a_0,  \ldots , a_m   \right)     \right]  $  
and  $  g(u)  \in  D    \subseteq  V^{  \star }  \setminus     \cup   \sigma   \left[  \Delta  \left(  \sigma  \,  ;  \,     b_0,  \ldots , b_m   \right)     \right]  $  
\end{enumerate}
This shows that (**) holds.

We consider (II). We have 
\[
 \tau (z)  =     \left(    \sigma  \left(   w^{  \star }  \right)   \setminus  A_z  \right)   \cup   B_z     
 \     \mbox{  and   }     \   
 \tau  \left(   g(z) \right)   =   
 \left(    \sigma  \left(   g \left( w^{  \star }   \right)   \right)   \setminus   g  \left[ A_z   \right]  \right)   \cup    g  \left[ B_z   \right]  
 \       .
 \]
Since  $g$  is one- to-one
\[
u  \in A_z   \   \Leftrightarrow   \  g(u)  \in    g  \left[ A_z   \right] 
\      \       \      \mbox{  and   }      \        \        \  
u  \in B_z   \   \Leftrightarrow   \  g(u)  \in    g  \left[ B_z   \right] 
\      .
\]
Hence, to show that  (**) holds, it suffices to show that  
$  u  \in    \sigma  \left(   w^{  \star }  \right)    $  if and only if   $  g(u)  \in   \sigma  \left(   g \left( w^{  \star }   \right)    \right)    \,   $. 
But this holds by (I). 
This completes the proof. 
\end{proof}

\begin{lemma}
$  \mathcal{G} $ is a countable family of dense subsets of  $  \mathbb{P} $.
\end{lemma}
\begin{proof}

Consider a set    $  \mathcal{G}    \left(    (a_0,  b_0 ) ,  \ldots , (b_m , b_m ) \,  ;   \,   a_{ m+1}   \right) $ in $  \mathcal{G}_2 $. 
We need to show that it is dense. 
So, pick $   \sigma  \in   \mathbb{P} $. 
We  need to show that there exists  $  \tau  \in  \mathbb{P} $  such that 
  $   \sigma  \subseteq  \tau   \in   \mathcal{G}   \left(    (a_0,  b_0 ) ,  \ldots , (a_m , b_m ) \,  ;   \,   a_{ m+1}   \right) $.
Recall that $  \tau $ needs to satisfies one of the following: 
   \begin{enumerate}
\item   There exists     $  b_{  m+1 }   \in   V^{  \star }  $ such that   
$  (a_0,  b_0 ) $,  $\ldots $,   $(a_m , b_m )  $,      $ (a_{ m+1 } ,  b_{ m+1 } ) $  is  $ \mathsf{Good}_2 \left(   \tau  \right)   $.

\item There is no   $  \tau^{  \prime  }    \in  \mathbb{P}   $  such that 
$  \tau  \subseteq   \tau^{  \prime  }   $   and 
  $  (a_0,  b_0 )   $,   $\ldots   $,  $ (a_m , b_m ) $  is  $ \mathsf{Good}_2  \left(   \tau^{  \prime  }    \right)   $.
 \end{enumerate} 
We have two cases: 
\begin{enumerate}
\item[\textup{(i)}] There is no   $ \mathbb{P}   \ni \sigma^{  \prime } \supseteq  \sigma $  such that 
  $  (a_0,  b_0 ) ,  \ldots , (a_m , b_m ) $  is  $ \mathsf{Good}_2  \left(   \sigma^{  \prime }  \right)   $.

\item[\textup{(ii)}]  There exists    $   \mathbb{P}   \ni    \sigma^{  \prime } \supseteq  \sigma $  such that 
  $  (a_0,  b_0 ) ,  \ldots , (a_m , b_m ) $  is  $ \mathsf{Good}_2 \left(   \sigma^{  \prime }  \right)   $. 
  \end{enumerate}
In case of (i), we can let  $  \tau :=  \sigma $. 
We consider case (ii). 
Let     $ \sigma^{  \prime } \supseteq  \sigma $  be such that 
  $  (a_0,  b_0 ) ,  \ldots , (a_m , b_m ) $  is  $ \mathsf{Good}_2  \left(   \sigma^{  \prime }  \right)   $.
  By Lemma    \ref{ZerothLemmaGood}, 
  there exists  $ \mathbb{P} \ni     \sigma^{  \prime  \prime }   \supseteq   \sigma^{  \prime }  $  such that  
  $  a_0,  \ldots , a_m , a_{ m+1}  $  is  $ \mathsf{Good}_1  \left(   \sigma^{  \prime } \right)   $. 
  By Lemma  \ref{FirstLemmaGood}, 
    $  (a_0,  b_0 ) ,  \ldots , (a_m , b_m ) $  is  $ \mathsf{Good}_2  \left(   \sigma^{  \prime  \prime }  \right)   $.
   By Lemma     \ref{SectionWDExtensionLemma}, 
there exists   $  \mathbb{P} \ni   \tau \supseteq  \sigma^{  \prime  \prime  }   $  and  $  b_{  m+1}  \in V^{  \star }  $  such that 
$  (a_0,  b_0 )  $,    $   \ldots   $,    $ (a_m,  b_m )  $,    $ (a_{ m+1 } ,  b_{ m+1 } ) $  is  $ \mathsf{Good}_2 \left(   \tau  \right)   $. 
This completes the proof.
\end{proof}

\section{Proof}

Since  $  \mathcal{D} $ is a countable family of dense subsets of  $  \mathbb{P} $, 
a $   \mathcal{D} $-generic ideal  $G$ can be constructed by recursion. 
The following theorem  then    completes the proof of Theorem  \ref{maintheorem}.

\begin{theorem}
Let  $ G \subseteq \mathbb{P} $ be  a  $  \mathcal{D} $-generic ideal.  
 Let  $ ( \cdot )^{  \star }    :=   \cup  G  $. 
  Let  $    \mathcal{V}^{ \star }  $ denote the structure  $  \left(  V^{  \star }   ,  \in^{  \star }     \right)   $
  where  for all  $  i, j  \in  \mathbb{Z} $,   we have   $   \mathsf{c}_i   \in^{  \star }   \mathsf{c}_j  $ if and only if  $  \mathsf{c}_i    \in   \mathsf{c}_j^{  \star }  $. 
Then, the following holds: 
\begin{enumerate}
\item  $  (  \cdot )^{  \star }   :  V^{  \star }     \to        \mathcal{X}   $  is a bijection, and hence  
$ \mathcal{V}^{ \star }    \models      \mathsf{WD} +   \mathsf{EXT} +  \mathsf{BU}    +  \mathsf{BI} $. 

\item  For all  $   k, \ell  \in  \mathbb{Z} $,   there exists an automorphism  
$  F^k_{  \ell  }   :    \mathcal{V}^{  \star }    \to      \mathcal{V}^{  \star }    $  such that   
$  F^k_{  \ell  }    \left(   \mathsf{c}_k  \right)   =    \mathsf{c}_{  \ell   }       \,      $. 
  \end{enumerate}
\end{theorem}
\begin{proof}

Lemma  \ref{FirstBasicLemma}  and  Lemma \ref{SecondBasicLemma}  show that    (1) holds. 
We show that (2) holds. 
Pick   $   k, \ell  \in  \mathbb{Z} $.   
We construct    $  F^k_{  \ell  }    $  by  a back-and-forth argument. 
Since  $  V^{  \star }  $  is countable,  fix   a   one-to-one enumeration   
 $   \left(  w_n  :   \   n   \in  \omega \setminus  \lbrace 0  \rbrace   \,   \right)  $ of  $  V^{  \star }  $. 
We construct an increasing sequence  $  \left(   g_n :    \     n  \in   \omega   \,  \right)  $ of finite partial    one-to-one maps  $    V^{  \star }   \to   V^{  \star }   $  such that
the following holds for all $  n  \in  \omega  $:  
\begin{enumerate}
\item  $  g_0   =  \lbrace   \left(  \mathsf{c}_k  ,  \mathsf{c}_{  \ell  }   \right)    \rbrace  $. 

\item  If  $n  >0 $  is even,  then  $  g_n  \left(  w_{  \frac{n}{2} }   \right)  $  is defined.

\item  If  $n$  is odd,  then  $    w_{  \frac{n+1}{2}  }    $  is  in the image of  $  g_n  $.

\item There exists  $  \sigma_n   \in  G  $  such that   any enumeration of the graph of  $  g_n  $   gives a 
$  \mathsf{Good}_2  \left(  \sigma_n  \right)  $  sequence.  
\end{enumerate}
First,  let us observe that  it follows from (4) that  each   each  $g_n $  is a partial embedding   $   \mathcal{V}^{  \star }   \to      \mathcal{V}^{  \star }   $, 
and the map    $   F^k_{  \ell  }   :=  \bigcup_{ n  \in  \omega  }   g_n     $  is thus an automorphis of     $   \mathcal{V}^{  \star }   $
since it is  a bijection by (2)-(3).  
Indeed,  let  $  (x_0, y_0) $,   $  \ldots  $,  $(x_k, y_k ) $ be an enumeration of the graph of  $  g_n $. 
By assumption,  the sequence  $  (x_0, y_0) $,   $  \ldots  $,  $(x_k, y_k ) $   is  $  \mathsf{Good}_2  \left(   \sigma_n  \right)      \,   $.
This means   in particular  that  $   \lbrace  x_0,  y_0 ,  \ldots  , x_k,  y_k   \rbrace   \subseteq    \mathsf{dom} \left(  \sigma_n   \right)   $  
and     there exists  an isomorphism 
\[
f   :     \left(    \Delta \left(  \tau  \, ;   \vec{x}   \right)    \,  ,   \,   \in^{  \sigma_n }    \right)     \to  
\left(    \Delta \left(  \tau  \, ;   \vec{y}   \right)    \,  ,   \,   \in^{  \sigma_n }    \right)    
\]
such that  $  f   \left(  x_i   \right)  =  y_i  $  for all $  i  \in  \lbrace 0 ,  1,  \ldots  ,  k  \rbrace  $. 
In particular,    $g_n $  is an isomorphism 
\[
g_n :     \left(      \lbrace x_0, x_1,  \ldots , x_k  \rbrace   \,  ,   \,   \in^{  \sigma_n }    \right)     \to  
\left(    \lbrace y_0, y_1,  \ldots , y_k  \rbrace    \,  ,   \,   \in^{  \sigma_n }    \right)    
\          .
\]
Recall that  
\[
\in^{  \sigma_n }  
   \,       =      \,    
  \left\{    (u, v)  \in  V^{  \star }  \times   \mathsf{dom} \left(  \sigma_n   \right)  :    \        u   \in  \sigma_n   (u)      \,         \right\}  
\        .  
\]
Since   $  \in^{  \star }  \,   =   \,   \bigcup_{    \tau   \in  G    }     \in^{  \tau   }      $
and    $   \lbrace  x_0,  y_0 ,  \ldots  , x_k,  y_k   \rbrace   \subseteq    \mathsf{dom} \left(  \sigma_n   \right)   $, 
the map    $  g_n  $  is  an isomorphism   
\[
g_n :     \left(      \lbrace x_0, x_1,  \ldots , x_k  \rbrace   \,  ,   \,   \in^{  \star  }    \right)     \to  
 \left(      \lbrace y_0, y_1,  \ldots , y_k  \rbrace   \,  ,   \,   \in^{  \star  }    \right)    
\          .
\] 
This completes the proof that  $g_n $  is a partial embedding     $   \mathcal{V}^{  \star }   \to      \mathcal{V}^{  \star }       \,   $.

Next, we show that  $g_0 $  satisfies  (1)-(4). 
We need to show that   $g_0 $  satisfies  (4). 
Since  $  \mathcal{G} $  consists of dense subsets of   $  \mathbb{P} $,   there exists  $  \tau_0,  \tau_1  \in  \mathbb{P} $  such that  
\[
\tau_0  \in G  \cap  \mathcal{G} \left(  \emptyset ;  \mathsf{c}_k    \right)  
\      \mbox{  and   }   \   
\tau_1  \in G  \cap  \mathcal{G} \left(  \emptyset ;  \mathsf{c}_{  \ell  }    \right)  
\          .
\]
Since  the empty sequence is   $  \mathsf{Good}_2  \left(  \sigma \right)  $   for  all $  \sigma  \in  \mathbb{P}  $, 
it follows from how   $   \mathcal{G} \left(  \emptyset ;  \mathsf{c}_k    \right)   $  
and   $   \mathcal{G} \left(  \emptyset ;  \mathsf{c}_{  \ell  }     \right)   $  are defined that 
there exist  $   x, y  \in  V^{  \star }   $  such that  
$  \left(   \mathsf{c}_k   ,   x  \right)  $  is    $  \mathsf{Good}_2  \left(  \tau_0  \right)  $
and 
$  \left(   \mathsf{c}_{  \ell }   ,   y  \right)  $  is    $  \mathsf{Good}_2  \left(  \tau_1  \right)  $. 
Since  $G$ is a  $  \mathcal{D} $-generic ideal,  
there exists $  \sigma_0  \in  D  $  such that  $  \tau_0  \subseteq   \sigma_0  $  and  $  \tau_1  \subseteq   \sigma_0  $. 
By Lemma \ref{FirstLemmaGood},  
$  \left(   \mathsf{c}_k   ,   x  \right)  $  is    $  \mathsf{Good}_2  \left(  \sigma_0  \right)  $
and 
$  \left(   \mathsf{c}_{  \ell }   ,   y  \right)  $  is    $  \mathsf{Good}_2  \left(  \sigma_0  \right)  $. 
In particular, 
the one-element sequence    $    \mathsf{c}_k  $  is   $  \mathsf{Good}_1  \left(  \sigma_0  \right)  $,  
and  the one-element sequence    $    \mathsf{c}_{  \ell }  $  is   $  \mathsf{Good}_1  \left(  \sigma_0  \right)     \,    $. 
By Lemma  \ref{SecondLemmaGood}, 
$  \left(   \mathsf{c}_{  k }   ,    \mathsf{c}_{  \ell }   \right)  $  is    $  \mathsf{Good}_2  \left(  \sigma_0  \right)        \,   $. 
Thus,  $g_0$  satisfies (1)-(4).

Finally,  assume  $  g_0,  \ldots ,  g_n  $  have been defined. 
We  show how to  define   $ g_{ n+1}  $. 
We assume  $ n  +1  $  is even; the other case is symmetric
since a sequence   $ (x_0,  y_0)  $,  $  \ldots  $,  $(x_j , y_j )  $  is   $  \mathsf{Good}_2  \left(  \sigma \right)  $  
if and only if     $ (y_0,  x_0)  $,  $  \ldots  $,  $(y_j , x_j )  $  is   $  \mathsf{Good}_2  \left(  \sigma \right)  $.
Let  $  (a_0,  b_0 ) $,  $  \ldots  $,   $  (a_{ m} ,  b_{ m }  )  $ be an enumeration of the graph of  $  g_n  $. 
Let  $  a_{ m+1}  :=     w_{  \frac{n+1}{2} }   $. 
By assumption,  there exists  $  \sigma_n  \in  G  $  such that    $  (a_0,  b_0 ) $,  $  \ldots  $,   $  (a_{ m} ,  b_{ m}  )  $ is 
 $  \mathsf{Good}_2  \left(  \sigma_n \right)    \,    $.  
Since   $  \mathcal{G} $  consists of dense subsets of   $  \mathbb{P} $,   there exists  $  \sigma_{ n+1}  \in  \mathbb{P} $  such that  
\[
\sigma_n   \subseteq   \sigma_{ n+1}  
\      \mbox{  and   }   \   
\sigma_{ n+1}   \in G  \cap  \mathcal{G}   \left(  (a_0, b_0) ,  \ldots  ,( a_m, b_m )   ;   a_{ m+1}  \right) 
\          .
\]
Since  
 $  (a_0,  b_0 ) $,  $  \ldots  $,   $  (a_{ m} ,  b_{ m}  )  $ is  $  \mathsf{Good}_2  \left(  \sigma_{ n}  \right)      $, 
Lemma \ref{FirstLemmaGood} tells us that  
     $  (a_0,  b_0 ) $,  $  \ldots  $,   $  (a_{ m} ,  b_{ m}  )  $ is   $  \mathsf{Good}_2  \left(  \sigma_{ n+1}  \right)      \,    $. 
 Hence, by how  the set  $  \mathcal{G}   \left(  (a_0, b_0) ,  \ldots  ,( a_m, b_m )   ;   a_{ m+1}  \right)   $  is defined, 
 there exists  $  b_{ m+1}  \in  V^{  \star  }  $  such that  
  $  (a_0,  b_0 ) $,  $  \ldots  $,   $  (a_{ m+1} ,  b_{ m+1}  )  $ is   $  \mathsf{Good}_2  \left(  \sigma_{ n+1}  \right)  $. 
 We extend  $  g_n $  to  $  g_{ n+1}  $  by setting  $  g_{ n+1}  \left(    a_{ m+1}   \right)  :=  b_{ m+1}   $. 
\end{proof}

\section*{ }


\begin{thebibliography}{}
\bibitem{Barnays1937}
Paul  Bernays,  A System of Axiomatic Set Theory - Part I,  The Journal of Symbolic Logic  2(1),  217-77   (1937).


\bibitem{Friedman2007}
Harvey Friedman, Interpretations, according to Tarski,   Nineteenth Annual Tarski Lectures (2007). 






 
 
 

\bibitem{Pudlak1983}
Pavel Pudlák,  Some Prime Elements in the Lattice of Interpretability Types, 
 Transactions of the American Mathematical Society,  280(1),   255–275   (1983).





\bibitem{Visser2010}
Albert Visser,  What is sequentiality? in: New Studies in Weak Arithmetics, edited by P. Cégielski, C. Cornaros, and C. Dimitracopoulos, CSLI Lecture Notes Vol.~211,  pp. 229–269   (2013).





\bibitem{Visser2019}
Albert   Visser,    The small-is-very-small principle,   Mathematical  Logic Quarterly,  65,  453-478   (2019). 


\end{thebibliography}
\end{document}